\documentclass[12pt]{amsart}

\usepackage{amssymb,amsthm}
\usepackage{amsfonts,cmtiup,comment,stmaryrd}

\newtheorem{theorem}{Theorem}[section]
\newtheorem{lemma}[theorem]{Lemma}
\newtheorem{proposition}[theorem]{Proposition}

\theoremstyle{definition}
\newtheorem*{definition}{Definition}

\newtheorem*{definition*}{Definition}

\textwidth=16cm \textheight=23cm \numberwithin{equation}{section}

\begin{document}
\title{Locally finite groups containing a $2$-element\\ with Chernikov centralizer}

\author{E. I. Khukhro}
\address{Sobolev Institute of Mathematics, Novosibirsk, 630\,090,
Russia,\newline and University of Lincoln, U.K.} \email{khukhro@yahoo.co.uk}
\thanks{This work was supported by CNPq-Brazil. The first author thanks  CNPq-Brazil
and the University of Brasilia for support and hospitality that he
enjoyed during his visits to Brasilia.}

\author{N. Yu. Makarenko}
\address{Sobolev Institute of Mathematics, Novosibirsk, 630\,090,
Russia}\email{natalia\_makarenko@yahoo.fr} \thanks{The second
author was supported  by the Russian Science Foundation, project no. 14-21-00065}

\author{P. Shumyatsky}
\address{Department of Mathematics, University of Brasilia, DF~70910-900, Brazil}
\email{pavel@unb.br}


\begin{abstract}
Suppose that a locally finite group $G$ has a $2$-element $g$
with Chernikov centralizer. 
It is proved that if the involution in $\langle g\rangle$ has nilpotent centralizer,  then $G$ has a soluble subgroup of finite index.
\end{abstract}

\maketitle
\hfill {\sl to the memory of Brian Hartley (1939--1994)}

\section{Introduction}

Studying groups with restrictions on centralizers is one of the main avenues of group theory.
One of the well-known problems  in this area is  studying locally finite groups~$G$ containing an element $g$ with Chernikov centralizer $C_G(g)$. A Chernikov group is a finite extension of a direct product of finitely many quasicyclic $p$-groups $C_{p^{\infty}}$ for various primes~$p$. (Recall that Chernikov groups
are precisely the locally finite groups with minimal condition on subgroups, as proved independently by Kegel and Wehrfritz \cite{ke-we} and Shunkov \cite{shun}.)   Hartley \cite{ha88} proved that if $g$ is of prime-power order, then $G$ has a locally
soluble subgroup of finite index. Earlier the case where  $g$ has order $2$ was handled by Asar \cite{asar}, and the case of $g$ of arbitrary prime order by Turau \cite{tura}. Both Hartley's and Turau's works depend on the classification of finite simple groups.

Hartley posed two problems in  \cite{ha88}. One problem is whether the above results can be extended to arbitrary order of $g$ --- and so far there have been no advances in this direction. Another problem is whether one can replace ``locally soluble'' by ``soluble'' in such results. This problem was solved in the positive for an element of prime order and order 4  in \cite{shu-p} and \cite{shu-4}, respectively.

In view of the difficulty of the above problems, it makes sense to study them under certain additional conditions. In the present paper we extend the result of \cite{shu-4} for elements of order~$2^n$. Our main result is the following theorem.

\begin{theorem}\label{31} Suppose that a locally finite group $G$ contains an  element $g$ of order $2^n$ with Chernikov centralizer and the centralizer
 of the involution $g ^{2^{n-1}}$ is nilpotent. Then $G$ has a  soluble subgroup of finite index.
\end{theorem}

A special case of a Chernikov group is a finite group, and the results on Chernikov centralizers often use the corresponding results on finite centralizers. The proof of Theorem~\ref{31} also uses our recent result of that kind, which we state for further references. We use abbreviation, say, ``$(a,b,\dots )$-bounded'' for
``bounded above in terms of  $a, b,\dots
$ only''.

\begin{theorem}[{\cite[Corollary~1.3]{khu-mak-shu}}]\label{t1}
 Suppose that a locally finite group $G$ contains an
element $g$ of order $2^n$ with finite centralizer of order
$m=|C_G(\varphi )|$ such that the centralizer
of the involution $g ^{2^{n-1}}$ is nilpotent
of class~$c$.  Then $G$ has a soluble characteristic subgroup of finite
$(m,n,c)$-bounded index that has $(n,c)$-bounded derived length.
\end{theorem}

The proof in \cite{khu-mak-shu} is mainly about nilpotent groups; reduction to soluble groups was provided earlier by Hartley's `generalized Brauer--Fowler theorem' in \cite{ha92} based on the classification, and reduction to nilpotent groups by the Hartley--Turau  theorem \cite{ha-tu}. The useful property that the subgroup can be chosen to be  characteristic follows from \cite{khu-mak-char}.

 Theorem~\ref{t1} in turn is a strengthening of  Theorem~1.10 in \cite{shu01}, which under the same hypotheses was giving a soluble subgroup of bounded index whose derived length was bounded in terms of $m,n,c$. It is elimination of the dependence of the derived length on $m$ in Theorem~\ref{t1} that made  possible its application in the proof of Theorem~\ref{31}. Similarly,  the aforementioned theorems
in \cite{shu-p,shu-4} on locally finite groups containing an element of prime order $p$ or of order $4$ with Chernikov centralizer used the results with `strong' bounds on   finite groups containing an element of order $p$ or $4$ with finite centralizer. Namely,  in the case of finite centralizer such a group contains a subgroup of bounded index that is nilpotent (or soluble) of nilpotency class (or derived length) bounded in terms of the order of the element only. This was proved for nilpotent groups in   \cite{khu85,khu90, mak93, mak-khu06} generalizing works on fixed-point-free automorphisms \cite{hi, kov, kr, kr-ko}, and the reduction  to nilpotent groups followed from \cite{ fon76, ha-me, ha-tu,tho59}.

\section{Preliminaries}

Given a group $G$ and subsets
$X,Y\subseteq  G$, we denote by $[X,Y]$ the
subgroup generated by all commutators $[x,y]$, where $x\in X$ and $ y\in Y$. It
is well-known that if $X,Y\leqslant G$ are subgroups, then $[X,Y]$ is a normal subgroup of $\langle X,Y\rangle$. Furthermore, if $X$ is a subgroup and $a$ is an element normalizing~$X$, then $[X,a]=[X,\langle a\rangle ]$ is also a normal subgroup of $\langle X,a\rangle$. We denote $[[X,Y],Y]$ by $[X,Y,Y]$.

When a group $A$ acts on a group $G$ by automorphisms (not necessarily faithfully), we regard $G$ and $A$ as subgroups of the semidirect product~$GA$. In accordance with the above notation, $[G,A]$ is the subgroup
of $G$ generated by all elements of the form $g^{-1}g^a$, where $g\in G$
and $a\in A$. Then $[G,A]$ is always an $A$-invariant
normal subgroup of~$G$. The following lemma will be very
helpful. The proof can be easily deduced from the
(well-known) corresponding facts about coprime actions of finite groups.

\begin{lemma}\label{3} Let $A$ be a finite $\pi$-group
of automorphisms of a locally finite group~$G$.
If $N$ is an $A$-invariant normal
$\pi'$-subgroup of~$G$, then $C_{G/N}(A)=C_G(A)N/N$.
\end{lemma}

For a group $G$ let $F(G)$ denote the Hirsch--Plotkin radical of~$G$, which is
the largest locally nilpotent normal subgroup of~$G$. We define $F_0(G)=1$ and let $F_{i+1}(G)$ be the full inverse image of $F(G/F_{i}(G))$ for $i=0,1,\dots$. The group $G$ is said to be of finite Hirsch--Plotkin height $m$ if $G=F_m(G)$ for some integer $m$ and $m$ is the least such integer. In this case we write $h(G)=m$. Of
course, if $G$ is a finite group, then the Hirsch--Plotkin radical $F(G)$ is
just the Fitting subgroup of~$G$. There are many results bounding the
Fitting height of a finite soluble group. In particular, we will use a special case of the
classical result of  Thompson \cite{tho64}  that if $G$ is a finite soluble group admitting a soluble group of automorphisms $A $ of coprime order, then $h(G)$ is bounded in terms of $h(C_G(A))$ and the number of prime factors in $|A|$ only. (Numerous subsequent papers improved the bounds, first linear ones were obtained by  Kurzweil
\cite{kurz}, and best-possible by Turull \cite{tu}).   The standard inverse limit argument as in
\cite[p. 54]{ke-we-b} extends this result to periodic locally soluble groups
$G$ having no $|A|$-torsion. In fact, we only need a special case of this theorem when $A$ is cyclic, stated as the following proposition.

\begin{proposition}\label{17} If a locally finite $2'$-group $G$ admits an
automorphism $\varphi$ of order $2^n$ such that $C_G(\varphi)$ is Chernikov, then
$G$ has finite Hirsch--Plotkin height.
\end{proposition}

A locally finite group $G$ is said to satisfy min-$p$ if every non-empty
family of $p$-subgroups of $G$ partially ordered by inclusion possesses a minimal member. We will  use the following helpful proposition \cite[Corollary 3.2]{ke-we-b}.

\begin{proposition}\label{7} If a locally finite
group $G$ contains elements of order~$p$, then
$G$ satisfies min-$p$ if and only if there is a
$p$-element $g\in G$ such that $C_G(g)$ satisfies
min-$p$.
\end{proposition}

Given a locally finite group~$G$,
we denote by $O_{p'}(G)$ the largest normal $p'$-subgroup
of~$G$.  Theorem 3.17 in
\cite{ke-we-b} tells us that if $G$ is a periodic
locally soluble group
satisfying min-$p$, then $G/O_{p'}(G)$ is Chernikov. Together with Proposition~\ref{7}, this implies the following.

\begin{proposition}\label{p-sol}
If $G$ is a
periodic 
locally soluble group admitting an automorphism of finite $2$-power order with Chernikov centralizer, then $G/O_{2'}(G)$ is a Chernikov group.
\end{proposition}

Let $C$ be a Chernikov group. A subgroup $D$ of finite index in $C$ that is a
direct product of finitely many groups of type $C_{p^\infty}$ (possibly
for various primes $p$) is of course unique in~$C$. Let $|C:D|=i$ and let $D$ be
a direct product of precisely $j$ groups of type~$C_{p^\infty}$.

\begin{definition}
We call the ordered pair $(j,i)$ the \emph{size} of~$C$.
\end{definition}

The set of all pairs
$(j,i)$ is endowed with the lexicographic order. It is easy to check that if
$H$ is a proper subgroup of~$C$, then the size of $H$ is necessarily strictly
smaller than the size of~$C$. Also, if $N$ is an \emph{infinite} normal subgroup of~$C$, then
the size of $C/N$ is strictly smaller than that of~$C$.  (Note, however, that if $N$ is finite, then the
size of $C/N$ may be equal to that of~$C$.) We shall use these properties of size of Chernikov subgroups without special references.

\section{Proof of the main theorem}

We are now ready to prove our main result, which we restate here in terms of automorphisms.

\begin{theorem} Let $G$ be a locally finite group admitting
an automorphism $\varphi$ of order $2^n$ such that $C_G(\varphi)$ is Chernikov and $C_G(\varphi^{2^{n-1}})$ is nilpotent. Then $G$ has a soluble subgroup of finite index.
\end{theorem}

\begin{proof} The group $G$ has a locally soluble subgroup of finite index by Hartley's theorem  \cite[Theorem~4]{ha88}, so without
loss of generality we can assume that $G$ is locally soluble. Then
$G/O_{2'}(G)$ is a Chernikov group by Proposition~\ref{p-sol}, so we can assume that $G$ is a $2'$-group.   Then the group $G$ has finite Hirsch--Plotkin height by
Proposition \ref{17}. Since $C_G(\varphi )$ covers the centralizers of the induced automorphism $\varphi$ in every invariant section by Lemma~\ref{3}, it follows that  if the theorem is false, then a counter-example to the theorem can be found among locally nilpotent sections of~$G$. Therefore without loss of generality we can additionally assume that $G$ is locally nilpotent.

Thus, arguing by contradiction we choose a locally nilpotent group $G$ that is a counter-example to the theorem such that the size of
$C=C_G(\varphi)$ is as small as possible. In view of Theorem~\ref{t1} it is clear that $C$ is infinite. Let $G^{(i)}$ denote the $i$th term of the derived series of~$G$. Each one of the groups $G^{(i)}$ provides a counter-example to the theorem. Indeed, if $G^{(i)}$ has a soluble subgroup of finite index, then $G^{(i)}$ itself is soluble since $G$ is locally nilpotent.  This gives solubility of~$G$, a contradiction.
Since the size of $C$ is as small as possible, it follows that $C\leqslant G^{(i)}$ for each~$i$, that is, $C$ is contained in the intersection of the~$G^{(i)}$. According to Theorem~1.10 in \cite{shu01} (or Theorem~\ref{t1}), there exists $k$ such that whenever $\varphi$ acts fixed-point-freely on the quotient $G/N$ by a normal $\varphi$-invariant subgroup~$N$, the derived length of $G/N$ is at most~$k$. On the other hand, Lemma \ref{3} shows that $\varphi$ acts fixed-point-freely on each quotient $G/G^{(i)}$. It follows that $G^{(k)}=G^{(k+1)}$. Thus, considering $G^{(k)}$ in place of $G$ we can make the additional assumption that
\begin{itemize}
\item[(1)\;\;] $ G=G'$.
\end{itemize}

Next, we can assume that the following holds.
\begin{itemize}
\item[(2)\;\;] If $N$ is a proper normal $\varphi$-invariant subgroup of~$G$, then the
centralizer $C_N(\varphi)$ is finite.
\end{itemize}
Indeed, suppose that $C_N(\varphi)$ is infinite. Since the size of $C$
is as small as possible, we deduce that $G/N$ is soluble, a contradiction with (1).

Let $A$ be a quasicyclic $p$-subgroup of $C$ for some prime~$p$. We claim that
\begin{itemize}
\item[(3)\;\;] $[G,A]=G$.
\end{itemize}
Indeed, $[G,A]A$ is a normal $\varphi$-invariant subgroup of $G$ having
an infinite intersection with~$C$. Now (2) implies that $[G,A]A=G$. Putting
this together with $G=G'$ we conclude that $[G,A]=G$.

We will use the fact that if $x,y\in A$, then necessarily either
$\langle x\rangle\leqslant\langle y\rangle$ or
$\langle y\rangle\leqslant\langle x\rangle$. It follows that
\begin{itemize}
\item[(4)\;\;] for any finite subset $S$ of $A$ we can choose $a\in S$ such that
$[N,S]=[N,a]$ for any normal subgroup~$N$.
\end{itemize}

Let $T$ be a finite subset of~$G$. Since $G=[G,A]$, every element of $T$
can be written as a product of finitely many commutators $[g_i,a_i]$
with $g_i\in G,a_i\in A$. In turn, each of the elements $g_i$ can be written
as a similar product of $[h_j,b_j]$ with $h_j\in G$ and $b_j\in A$. Let
$S$ be the set of all $a_i$ that appear in the commutators $[g_i,a_i]$
united with the set of all $b_j$ that appear in $[h_j,b_j]$. It is obvious
that $T\subseteq[G,S,S]$. Thus, we derive from (4) that
\begin{itemize}
\item[(5)\;\;] for any finite subset $T$ of $G$ there exists $a\in A$ such that
$T\subseteq [G,a,a]$.
\end{itemize}

Since the group $G$ is locally nilpotent,
\begin{itemize}
\item[(6)\;\;] $[G,g]$ is a proper subgroup for every  $g\in G$.
\end{itemize}
This property may be well known, but we could not find a reference, so we give a short proof here. Namely, we claim that $g\not\in [G,g]$ if $g\ne 1$. Indeed, otherwise $g$ is equal to a product of several commutators $[g,h_i]$ for finitely many elements $h_i\in G$, $i=1,\dots ,k$. Substituting  the same expression for $g$ into these commutators we obtain that $g$ is a product of commutators of weight at least 3 each involving the element~$g$. Repeatedly substituting in similar fashion expressions for $g$ we eventually obtain an expression for $g$ as a product of commutators in the same elements $ g,h_1,\dots ,h_k$ of weight exceeding the nilpotency class of the finitely generated subgroup  $\langle g,h_1,\dots ,h_k\rangle$, so that $g=1$.

We now consider the subgroup $[G,a]$ for $a\in A$. By (6) this is a proper normal subgroup of~$G$, which is obviously also
$\varphi$-invariant. So we conclude from (2) that $[G,a]$ has finite
intersection with $C_G(\varphi)$ for any $a\in A$. Therefore  by Theorem~\ref{t1} we obtain that $[G,a]$ contains a characteristic  subgroup $K_a$ of finite index that is soluble with derived length at most~$d$, where $d$ is independent of~$a$. Note that both $[G,a]$ and $K_a$ are $A$-invariant. Since $[G,a]/K_a$ is finite and since $A$ has no subgroups of finite index, it follows that $A$ centralizes the section $[G,a]/K_a$. Thus,
\begin{itemize}
\item[(7)\;\;] $[[G,a],A]\leqslant K_a$ for every $a\in A$.
\end{itemize}

Combining this with (5), we deduce that for an arbitrary finite subset
$T$ of $G$ there exists $a\in A$ such that $T\subseteq K_a$. Therefore any finite subset of $G$ generates a soluble subgroup with derived length at most~$d$. It follows that $G$ is soluble of derived length at most~$d$. This contradicts~(1). The proof is complete.
\end{proof}

\end{document}